\documentclass[a4paper,12pt]{amsart}

\usepackage[margin=1in]{geometry}

\usepackage{amsmath, amsthm, amssymb}
\usepackage{cite}
\usepackage[dvipdfmx]{graphicx}
\usepackage{bm}

\theoremstyle{definition}
\newtheorem{theorem}{Theorem}[section]
\newtheorem{lemma}[theorem]{Lemma}
\newtheorem{proposition}[theorem]{Proposition}
\newtheorem{fact}[theorem]{Fact} 
\newtheorem{observation}[theorem]{Observation} 

\theoremstyle{definition}
\newtheorem{definition}[theorem]{Definition}
\newtheorem{example}[theorem]{Example}
\newtheorem{pclaim}[theorem]{Claim}
\newtheorem*{ac}{Acknowledgments} 

\theoremstyle{remark}
\newtheorem{remark}[theorem]{Remark}

\newenvironment{rmenum}{
\begin{enumerate}

}
{\end{enumerate}}

\newcommand{\distgtf}[5]{\lambda(#4, #5; #3; #1, #2)}
\newcommand{\parcut}[2]{\delta_{#1}(#2)}

\newcommand{\critical}[1]{\mathcal{G}(#1)}

\newcommand{\fcr}[1]{\mathcal{L}(#1)} 
\newcommand{\fcgtr}[1]{\mathcal{L}^{+}(#1)}

\title{Constructive Characterization of Critical Bipartite Grafts}
\author{Nanao Kita}
\address{Furo-cho, Chikusa-ku, Nagoya, 464-8601, Japan}
\email{kita@math.nagoya-u.ac.jp}

\date{\today}

\begin{document}

\begin{abstract} 
Factor-critical graphs are a classical concept in matching theory 
 that constitute  an important component of  the Gallai-Edmonds canonical decomposition and Edmonds' algorithm for maximum matchings. 
Lov\'asz provided a constructive characterization of factor-critical graphs in terms of ear decompositions. 
This characterization has been a useful inductive tool for studying factor-critical graphs 
and also connects them with Edmonds' algorithm.

Joins in grafts, also known as $T$-joins in graphs, are a classical variant of matchings proposed in terms of parity. 
Minimum joins and grafts are generalizations of perfect matchings and graphs with perfect matchings, respectively. 
Accordingly, 
  graft analogues of fundamental concepts and results from matching theory, such as canonical decompositions, 
  will be fundamental contributions to the theory of minimum join. 
In this paper, we propose a new concept of {\em critical quasicombs}  
as a bipartite graft analogue of factor-critical graphs and 
provide a constructive characterization of critical quasicombs using a graft version of ear decompositions. 
This characterization can be considered as a bipartite graft analogue of Lov\'asz' result.  
From our results, the Dulmage-Mendelsohn canonical decomposition, 
originally a theory for bipartite graphs,  has been generalized for bipartite grafts. 
\end{abstract}

\maketitle

\section{Introduction}

\subsection{Joins and Grafts}

Joins in grafts, or $T$-joins in graphs, are a classical topic in combinatorics~\cite{lp1986, kv2008, schrijver2003}. 
Let $G$ be a graph, and let $T$ be a set of vertices.  A set $F$ of edges is a $T$-{\em join} of $G$ if every vertex in $T$ is connected to an odd number of edges from $F$, 
whereas other vertices are connected to an even number of edges from $F$.     
A $T$-join of graph $G$ is also referred to as a {\em join} of the pair $(G, T)$. 
The pair $(G, T)$ is called a {\em graft} if every connected component of $G$ has an even number of vertices from $T$.

Minimum joins in a graft are a generalization of $1$-factors  (perfect matchings) in a {\em factorizable} graph, that is, a graph with $1$-factors.  
A pair $(G, T)$ of a graph $G$ and a set of vertices $T$ has a join if and only if it is a graft.  
Hence, minimum joins in grafts are typically of interest. 
Let $(G, T)$ be a graft where $G$ is a factorizable graph and $T$ is the vertex set of $G$. 
Then, a set of edges is a minimum join of $(G, T)$ if and only if it is a $1$-factor of $G$.  
That is, grafts and minimum joins are generalizations of factorizable graphs and  $1$-factors, respectively. 

In addition, the minimum join problem on  grafts contains several types of routing problems on (undirected) graphs 
such as the shortest path problem between two vertices and the Chinese postman problem~\cite{lp1986, schrijver2003}. 
Joins on grafts are also closely related to  classical unsolved problems in graph theory 
such as Tutte's $4$-flow conjecture and the circuit double cover conjecture~\cite{schrijver2003}. 
Compared with $1$-factors, joins are rather poorly understood even though they have attracted considerable attention. 
Joins implicate significantly complicated and esoteric notions, as is often the case with generalized concepts involving parity.

\subsection{Toward  Canonical Decomposition Theory for Joins in Grafts} 
\subsubsection{Canonical Decompositions in Matching Theory} 

Canonical decompositions of a graph are a series of structure theorems that constitute the foundation of matching theory~\cite{lp1986}. 
Several types of canonical decompositions are known for matchings, 
such as the {\em Gallai-Edmonds}~\cite{gallai1964, edmonds1965}, {\em Kotzig-Lov\'asz}~\cite{kotzig1959a, kotzig1959b, kotzig1960, lovasz1972b, kitacathedral, DBLP:conf/isaac/Kita12}, {\em Dulmage-Mendelsohn}~\cite{dm1958, dm1959, dm1963}, and {\em cathedral decompositions}~\cite{kitacathedral, DBLP:conf/isaac/Kita12}. 
Each canonical decomposition has its own contexts or class of graphs where it can be effectively applied,   
but there are several common features. 
For a given graph, a canonical decomposition defines a uniquely determined partition, 
and uses this partition to characterize the structure of all maximum matchings in the graph. 
This feature makes  canonical decompositions powerful tools in studying matchings that are employed in many contexts. 
Following canonical decompositions, 
a concept from matching theory is said to be canonical if it is uniquely determined for a graph.

\subsubsection{Factor-Components and Canonical Decompositions for Factorizable Graphs} \label{sec:canonical:factorizable} 

Factor-components are the fundamental building blocks of a factorizable graph in the context of matching theory~\cite{lp1986, kitacathedral, DBLP:conf/isaac/Kita12}.  
A factorizable graph consists of disjoint subgraphs called {\em factor-components} and the edges between them. 
A set of edges is a $1$-factor if and only if it is a union of $1$-factors taken from each factor-component, 
and the factor-components form the most fine-grained set of subgraphs that satisfies this property.

The Dulmage-Mendelsohn and cathedral decompositions are canonical decompositions that are substantially applicable to 
bipartite and nonbipartite factorizable graphs, respectively;  
  these decompositions are  structure theorems that characterize the organization of a graph from its factor-components.    
In these decompositions,  the relationship between factor-components is characterized 
by defining a certain canonical binary relation between them  
and then proving that this relation is in fact a partial order. 
These partial orders are called the {\em Dulmage-Mendelsohn} and {\em cathedral orders}. 
\footnote{More precisely, the cathedral decomposition is a composite theory 
in which the cathedral order is a constituent. The {\em general} Kotzig-Lov\'asz decomposition is another constituent. } 

Note that the cathedral decomposition is not a generalization of the Dulmage-Mendelsohn decomposition. 
The former is by definition applicable to all graphs including bipartite graphs. 
However, it is substantially applicable to nonbipartite graphs because no distinct factor-components in a bipartite graph can be compatible 
with respect to the cathedral order.  
The Dulmage-Mendelsohn decomposition is established on the premise of  the two color classes, 
and can nontrivially apply to bipartite graphs.

\subsubsection{Extension to Grafts} 

Because minimum joins in grafts are a generalization of $1$-factors in factorizable graphs, 
 a canonical decomposition theory for grafts will  be the foundation for studying minimum joins.    
 The Kotzig-Lov\'asz, Dulmage-Mendelsohn, and cathedral decompositions are canonical decompositions for factorizable graphs.  
 Among them, however, only the Kotzig-Lov\'asz decomposition has been generalized for grafts~\cite{sebo1987dual, kita2017parity}.

Under the analogy and the observation from Section~\ref{sec:canonical:factorizable},  
generalizations of the Dulmage-Mendelsohn and cathedral decompositions for grafts should be structure theorems 
that  characterize the canonical relationship between factor-components in a bipartite and nonbipartite graft, respectively. 
The concept of factor-components can easily be generalized and defined for grafts~\cite{kita2017parity}. 
 Bipartite and  nonbipartite grafts are also easily defined as 
grafts whose graph part is bipartite and nonbipartite, respectively. 
Reasonable generalizations of the Dulmage-Mendelsohn and cathedral decompositions 
would define a canonical binary relation between factor-components 
and then prove that this relation is a partial order for bipartite and nonbipartite grafts, respectively.

\subsection{Factor-Critical Graphs}

\subsubsection{Basics} 

Factor-critical graphs are a classical concept in matching theory~\cite{lp1986, kv2008, schrijver2003}.  
A graph is {\em factor-critical} if, for every vertex, deleting it results in a factorizable graph. 
It is easily confirmed from the definition that no factor-critical graph can be bipartite. 
Factor-critical graphs were proposed and appeared in the Gallai-Edmonds decomposition when Gallai~\cite{gallai1964} discovered this decomposition. 
In this decomposition, first, a partition of the vertex set into three parts is defined.  
Then, it is proved that the first part, which is defined as the set of vertices that may not be covered by a maximum matching,  
induces a subgraph whose connected components are factor-critical. 
From the connection between the decomposition and Edmonds' algorithm for maximum matchings~\cite{edmonds1965, lp1986, kv2008, schrijver2003}, 
this also implies that factor-critical graphs are an important element of this algorithm. 
That is, the so-called blossoms are in fact factor-critical graphs. 
Because the Gallai-Edmonds decomposition is quite a useful theorem and  employed in numerous contexts in matching theory, 
factor-critical graphs are also involved in many situations. 

\subsubsection{Lov\'asz' Constructive Characterization of Factor-Critical Graphs} 

Lov\'asz~\cite{lovasz1972note, lp1986, kv2008, schrijver2003} discovered a constructive characterization of factor-critical graphs in terms of ear decompositions.  
An {\em ear decomposition} is a way of interpreting a graph as being constructed from a single vertex by sequentially adding paths or circuits called {\em ears}.  
Lov\'asz proved that a graph is factor-critical  
if and only if it admits an ear decomposition in which every ear is {\em odd}, that is, has an odd number of edges. 
This characterization of factor-critical graphs has served as a very useful inductive tool for proving properties of factor-critical graphs.   
It also connects factor-critical graphs and blossoms from Edmonds' algorithm. 

\subsection{Factor-Critical Graphs and Cathedral Decomposition} 

Factor-critical graphs also play an important role in the cathedral decomposition theory~\cite{kitacathedral, DBLP:conf/isaac/Kita12}.  
That is, the cathedral order is defined using factor-critical graphs. 
In a factorizable graph $G$, 
a factor-component $C$ is greater than a factor-component $D$ 
if the graph obtained from $G$ by contracting $D$ has a certain factor-critical subgraph that contains $C$. 
More precisely,  $C$ is greater than $D$ 
 if $G$ has factor-components  in which $C$ and $D$  are included 
 such that the subgraph of $G$ induced by the set of vertices from these factor-components 
 becomes  factor-critical after $D$ is contracted.   
The cathedral decomposition has been applied to derive further results in matching theory~\cite{kita2014alternative, kita2015graph,  kita2012canonical, DBLP:conf/cocoa/Kita13, DBLP:conf/cocoa/Kita17}   
such as new proofs of the tight cut lemma by Edmonds et al.~\cite{elp1982} and the saturated cathedral theorem by Lov\'asz~\cite{lovasz1972b} 
as well as 
a characterization of the family of maximal {\em barriers}~\cite{lp1986}, the dual of maximum matchings under the Tutte-Berge theorem, in general graphs.

\subsection{Our Aim and Results} 

\begin{figure} 
\includegraphics[width=.8\textwidth]{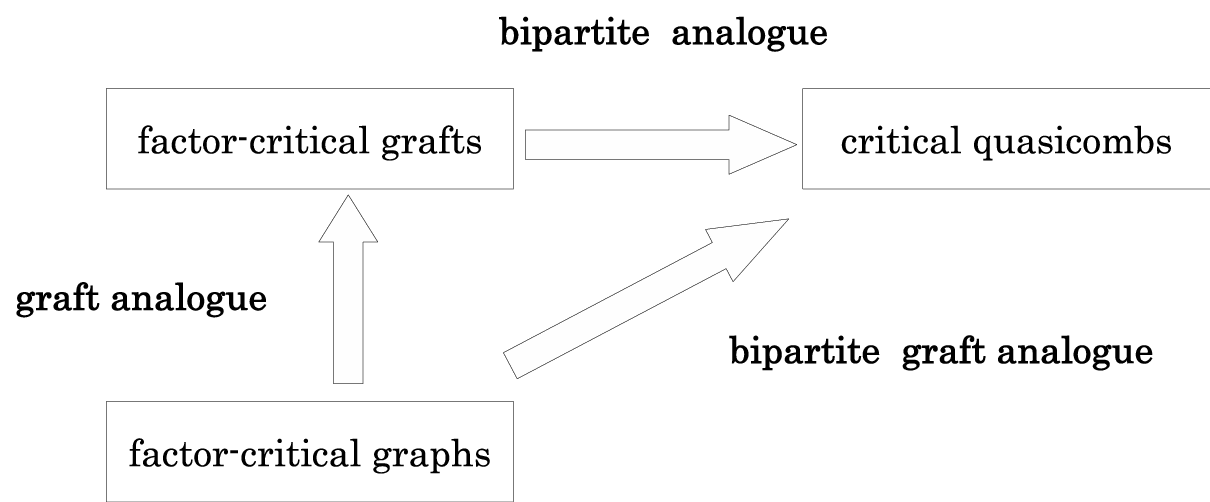} 
\caption{Relationships among factor-critical graph, factor-critical grafts, and critical quasicombs.} 
\label{fig:chart}
\end{figure}

We propose a new concept called {\em critical quasicomb} grafts that is a bipartite graft analogue of factor-critical graphs  
and provide a constructive characterization of these grafts 
 for the purpose of deriving a generalization of the Dulmage-Mendelsohn decomposition for bipartite grafts. 
From the analogy between grafts and factorizable graphs, 
we can naturally define {\em factor-critical grafts}, that is, a graft analogue of factor-critical graphs; see Figure~\ref{fig:chart}.  
Somewhat surprisingly, Seb\"o's distance theorem~\cite{DBLP:journals/jct/Sebo90} 
implies that a graft $(G, T)$ is factor-critical 
if and only if $G$ is a factor-critical graph and $T$ consists of all vertices except for one. 
As with factor-critical graphs, no factor-critical graft can be bipartite; every vertex in a factor-critical graft has a symmetrical role. 
In this paper, we further propose a bipartite analogue of factor-critical grafts: critical quasicombs. 
Criticality and bipartiteness are naturally and reasonably integrated in critical quasicombs.

Furthermore, we provide a constructive characterization of critical quasicombs that is an analogue of 
Lov\'asz' characterization of factor-critical graphs.  
First, we define the addition of grafts and the graft analogue of the concept of ears; 
these enable us to define ear decompositions of grafts. 
We then propose  types of ear grafts, which we term {\em effective} ear grafts, as the bipartite graft analogue of odd ears.  
We thus prove that a bipartite graft is a critical quasicomb  
if and only if it admits a graft ear decomposition in which every ear is effective. 
The proof is given from first principles.

\subsection{Consequence of Our Results} 

In our next paper~\cite{kita2021bipartite}, we use the results in this paper to derive a generalization of the Dulmage-Mendelsohn decomposition for bipartite grafts.  
This new decomposition characterizes the relationship between factor-components in a bipartite graft 
as a certain partial order.  
If a graft is a pair of a bipartite factorizable graph $G$ and  its vertex set, 
 then the generalized decomposition for this graft coincides with the original Dulmage-Mendelsohn decomposition for $G$.

This new decomposition can also be referred to as a bipartite graft analogue of the cathedral decomposition.\footnote{Note that there is an analogical relationship between the two, and one is not a generalization of the other. }
The  partial order between factor-components in a bipartite graft is similar to the cathedral order  
except that critical quasicombs replace the role of factor-critical graphs.   
That is,  in this new decomposition,  
factor-component $C$ is defined to be greater than factor-component $D$ 
if there are factor-components in which $C$ and $D$ are included 
such that the subgraph induced by the set of vertices from these factor-components becomes a critical quasicomb after $D$ is contracted. 
Then, it is proved that this binary relation is a partial order 
 where  our constructive characterization of critical quasicombs is used for the derivation.  
From this new decomposition, 
we can also understand the relationship between the original Dulmage-Mendelsohn and cathedral decompositions for factorizable graphs  
as the former being a bipartite analogue of the latter. 

\subsection{Organization of This Paper} 

Section~\ref{sec:definition}  provides the basic definitions and notations.  
Sections~\ref{sec:graft} and \ref{sec:distance}  explain the fundamentals regarding joins and grafts. 
Sections~\ref{sec:fc} and \ref{sec:fcgraft} 
 present preliminaries regarding factor-critical graphs and grafts, respectively. 
In Section~\ref{sec:bipartite},  
we introduce the concept of  bipartite grafts with an ordered bipartition to 
facilitate the subsequent discussions on  bipartite grafts. 
In Section~\ref{sec:comb}, 
we explain a class of bipartite grafts known as {\em combs} 
and an existing result to be used in Section~\ref{sec:crqcomb}. 
In Section~\ref{sec:quasicomb}, we introduce a new concept of {\em quasicomb} bipartite grafts.   
In Section~\ref{sec:crqcomb}, 
we define {\em critical} bipartite grafts, that is, the bipartite graft analogue of factor-critical graphs. 
Here, we show that critical bipartite grafts are in fact special quasicombs and call critical bipartite grafts as critical quasicombs onwards. 
In Section~\ref{sec:ear}, 
we introduce effective ear grafts as the components for constructing critical quasicombs. 
In Section~\ref{sec:main},  
our main theorem, namely, 
the constructive characterization for critical bipartite grafts is proved. 
Finally, Section~\ref{sec:conclusion} concludes the paper.

\section{Definitions} \label{sec:definition}

\subsection{Basic Notation} 
We mostly follow Schrijver~\cite{schrijver2003} for standard notation and definitions.  
In the following, we list exceptions or nonstandard definitions.  
For two sets $A$ and $B$, we denote the symmetric difference $(A\setminus B) \cup (B\setminus A)$ by $A\Delta B$. 
We often denote a singleton $\{x\}$ simply by $x$. 
For a graph $G$, we denote the vertex and edge sets of $G$ by $V(G)$ and $E(G)$, respectively.  
We consider multigraphs. For two vertices $u$ and $v$ from a graph, $uv$ denotes an edge between $u$ and $v$. 
We consider paths and circuits as graphs. 
That is, a circuit is a connected graph in which every vertex is of degree two, 
whereas a path is a connected graph in which every vertex is of degree two or less, and at least one vertex is of degree less than two. 
A graph with a single vertex and no edge is a path.   
For a path $P$ and vertices $x, y \in V(P)$, 
we denote by $xPy$ the subpath of $P$ with ends $x$ and $y$. 
We often treat a graph $G$ as if it is the set $V(G)$ of vertices. 
A perfect matching of a graph $G$ is a set $M$ of edges such that 
every vertex of $G$ is connected to exactly one edge from $M$.

In the remainder of this section, let $G$ be a graph unless stated otherwise. 
Let $X\subseteq V(G)$. 
The graph obtained by deleting $X$ is denoted by $G - X$.  
For $X, Y \subseteq V(G)$, 
we denote by $E_G[X, Y]$ the set of edges between $X$ and $Y$. 
The set $E_G[X, V(G)\setminus X]$ is denoted by $\parcut{G}{X}$.

Let $G_1$ and $G_2$ be subgraphs of $G$. 
The sum of $G_1$ and $G_2$ is denoted by $G_1 + G_2$. 
Let $F\subseteq E(G)$.  For a subgraph $H$ of $G$, 
$H + F$ denotes the graph obtained by adding $F$ to $H$. 
The subgraph of $G$ determined by $F$ is denoted by $G. F$. That is, $G. F = (V(G), F)$.

\subsection{Ears} 

Let $G$ be a graph, and let $X\subseteq V(G)$. 
Let $P$ be a subgraph of $G$ with $E(P)\neq \emptyset$. 
We call  $P$  a {\em round ear} relative to $X$ 
if $P$ is a path  whose vertices except for its ends are disjoint from $X$ or a circuit that shares exactly one vertex with $X$. 
We call  $P$ a {\em straight ear} relative to $X$ 
if $P$ is a path one of whose ends is the only vertex that $P$ shares with $X$. 
We call $P$ an {\em ear} if $P$ is either a round or straight ear. 
For an ear $P$ relative to $X$,   
vertices in $V(P)\cap X$ are called  {\em bonds } of $P$. 
We call the end of a straight ear that is not the bond the {\em free end}.

\section{Joins and Grafts} \label{sec:graft} 
\subsection{Basics} 

Let $(G, T)$ be a pair of a graph $G$ and a set $T\subseteq V(G)$. Set $T$ may be empty. 
A set $F\subseteq E(G)$ is a {\em join} of $(G, T)$ 
if $|\parcut{G}{v}\cap F|$ is odd for each $v\in T$ and is even for each $v\in V(G)\setminus T$. 
The pair $(G, T)$ is a {\em graft} if  $|V(C)\cap T|$ is even  for every connected component  $C$ of $G$. 

\begin{fact} \label{fact:graft} 
Let  $G$ be a graph,  and let $T\subseteq V(G)$. 
Then, $(G, T)$ has a join if and only if it is a graft. 
\end{fact} 

Joins in a graft are also referred to as $T$-joins in a graph. 
Under Fact~\ref{fact:graft}, we are typically interested in {\em minimum joins}, 
that is, joins with the minimum number of edges. 
We denote  by $\nu(G, T)$ the number of edges in a minimum join of a graft $(G, T)$. 
According to the next fact, minimum joins and  grafts are a generalization of perfect matchings and  graphs with perfect matchings~\cite{lp1986}.  
\begin{fact} 
Let $G$ be a graph with perfect matchings. 
Then, a set $M\subseteq E(G)$ is a perfect matching of $G$ 
if and only if it is a minimum join of the graft $(G, V(G))$.  
\end{fact} 

We often treat items of $G$ as items of $(G, T)$. 
For example, we call a path in $G$ a path in $(G, T)$. 
A graft $(G, T)$ is {\em bipartite} if $G$ is bipartite. 
We call sets $A$ and $B$ color classes of $(G, T)$ if these sets are color classes of $G$.

\subsection{Addition of Grafts} 

Let $(G_1, T_1)$ and $(G_2, T_2)$ be grafts. 
We define the sum of $(G_1, T_1)$ and $(G_2, T_2)$ as 
the graft $(G_1 + G_2, T_1\Delta T_2)$ and denote this graft by $(G_1, T_1)\oplus(G_2, T_2)$.

\begin{observation} \label{obs:addition}
If $F_1$ and $F_2$ are joins of grafts $(G_1, T_1)$ and $(G_2, T_2)$, respectively, 
then $F_1 \cup F_2$ is a join of the graft $(G_1, T_1)\oplus(G_2, T_2)$. 
\end{observation} 

Note that $F_1 \cup F_2$ may not be a minimum join even if $F_1$ and $F_2$ are minimum joins of $(G_1, T_1)$ and $(G_1, T_2)$. 

\subsection{Ear Grafts} 

Let $(G, T)$ and $(P, T')$ be grafts. 
Graft $(P, T')$ is an {\em ear graft} relative to $(G, T)$ if $P$ is an ear relative to $G$. 
Vertices of $P$ are called the {\em bonds} of $(P, T')$ if they are bonds of $P$. 
If $P$ is straight, the free end of $P$ is called the {\em free end} of $(P, T')$. 
We say that ear graft $(P, T')$ is {\em round} or {\em straight} if $P$ is round or straight, respectively.

\section{Edge Weighting and Distances on Grafts} \label{sec:distance} 

In this section, we explain a certain edge weight on a graft typically determined by a minimum join 
and the distances between vertices with respect to this edge weight.

\begin{definition} 
Let $(G, T)$ be a graft. Let $F \subseteq E(G)$. 
We define $w_F: E(G)\rightarrow \{1, -1\}$ as  
$w_F(e) = 1$ for $e\in E(G)\setminus F$ and $w_F(e) = -1$ for $e\in F$. 
For a subgraph $P$ of $G$, which is typically a path or circuit, 
 $w_F(P)$ denotes $ \Sigma_{ e \in E(P) } w_F(e)$, 
 and is referred to as the $F$-{\em weight} of $P$. 
For $u,v\in V(G)$, 
a path between $u$ and $v$ with the minimum $F$-weight is said to be {\em $F$-shortest} between $u$ and $v$.  
The $F$-weight of an $F$-shortest path between $u$ and $v$ is referred to as 
the $F$-{\em distance} between $u$ and $v$, 
and is denoted by $\distgtf{G}{T}{F}{u}{v}$. 
Regarding these terms, we sometimes omit the prefix ``$F$-'' if the meaning is apparent from the context. 
\end{definition}

The next proposition characterizes $F$-distances, where $F$ is a minimum join.  
It implies that the $F$-distance between vertices does not depend on the minimum join $F$. 

\begin{proposition}[Seb\"o~\cite{DBLP:journals/jct/Sebo90}] \label{prop:dist2invariant} 
Let $(G, T)$ be a graft, and let $F$ be a minimum join. 
Then, $\distgtf{G}{T}{F}{u}{v} = \nu(G, T\Delta \{u, v\}) - \nu(G, T)$ for every $u, v\in V(G)$. 
\end{proposition}

The next lemma provides a necessary and sufficient condition for a set $F$ to be a minimum join in terms of the $F$-weights of circuits. 
We will use this lemma to prove Theorem~\ref{thm:cr2qcomb} and Lemma~\ref{lem:ear2critical}.

\begin{lemma}[see also Seb\"o~\cite{DBLP:journals/jct/Sebo90}]  \label{lem:minimumjoin} 
Let $(G, T)$ be a graft, and let $F$ be a join of $(G, T)$. 
Then, $F$ is a minimum join of $(G, T)$ if and only if 
there is no circuit $C$ with $w_F(C) < 0$. 
\end{lemma}

\section{Factor-Critical Graphs}  \label{sec:fc} 

In this section, we explain the classical concept of factor-critical graphs and 
 their constructive characterization discovered by Lov\'asz~\cite{lovasz1972note, lp1986}. 
In Section~\ref{sec:fcgraft},  graft analogues of this concept and result are presented, 
whereas in Sections~\ref{sec:crqcomb} to \ref{sec:main},  bipartite graft analogues are presented.  

A graph is {\em factor-critical} if $G - v$ has a perfect matching for every $v\in V(G)$. 

\begin{definition} 
The family $\fcr{r}$ of graphs with vertex $r$ is defined as follows. 
\begin{rmenum} 
\item  $(\{r\},  \emptyset)$ is a member of $\fcr{r}$. 
\item Let $G$ be a member of $\fcr{r}$, and let $P$ be a round ear relative to $G$ such that $|E(P)|$ is odd. 
Then, $G + P$ is a member of $\fcr{r}$. 
\end{rmenum} 
\end{definition}

\begin{theorem}[Lov\'asz~\cite{lovasz1972note, lp1986}] \label{thm:lovasz} 
Let $G$ be a graph, and let $r \in V(G)$. 
Then, $G$ is factor-critical if and only if $G$ is a member of $\fcr{r}$. 
\end{theorem}

\begin{remark} 
It is easily observed from the definition or Theorem~\ref{thm:lovasz} 
that every factor-critical graph with at least one edge is nonbipartite. 
\end{remark}

The following characterizaion of factor-critical graphs is also well known 
and rather easily derived from a discussion of alternating paths. 
In fact, Proposition~\ref{prop:path2root} is used for deriving Theorem~\ref{thm:lovasz}.

\begin{proposition}[see Lov\'asz and Plummer~\cite{lp1986}]  \label{prop:path2root} 
Let $G$ be a graph, let $r \in V(G)$, and let $M\subseteq E(G)$ be a perfect matching of $G-r$. 
Then, $G$ is factor-critical 
if and only if, 
for every $v\in V(G)\setminus \{r\}$,  
there is a path in $G$ from $v$ to $r$ whose $M$-weight is zero. 
\end{proposition}

\section{Factor-Critical Grafts} \label{sec:fcgraft}

In this section, we introduce the concept of factor-critical grafts as a graft analogue of factor-critical graphs. 
A reasonable definition of factor-critical grafts is obvious from Seb\"o~\cite{DBLP:journals/jct/Sebo90}. 
We first provide the definition of factor-critical grafts and then explain its validity. 
We then convert Theorem~\ref{thm:lovasz} into a constructive characterization for factor-critical grafts.

The definition of factor-critical grafts is provided as follows.

\begin{definition} \label{def:fcgraft} 
A graft $(G, T)$ is {\em factor-critical} with {\em root} $r \in V(G)$ if 
 $\nu(G, T) = \nu(G, T\Delta \{r, v\})$ for every $v\in V(G)\setminus \{r\}$.  
\end{definition}

Somewhat surprisingly, a result of Seb\"o~\cite{DBLP:journals/jct/Sebo90} shows that factor-critical grafts can be characterized as follows. 

\begin{proposition}[Seb\"o~\cite{DBLP:journals/jct/Sebo90}] \label{prop:critical2char} 
The following two are equivalent for a graph $G$ and a set $T \subseteq V(G)$. 
\begin{rmenum} 
\item $(G, T)$ is a factor-critical graft with root $r \in V(G)$.    
\item Graph $G$ is  factor-critical, and $T = V(G) \setminus \{r\}$.  
\end{rmenum} 
\end{proposition}

The validity of the definition of factor-critical grafts, along with the impact of Proposition~\ref{prop:critical2char}, can be explained as follows. 
It is easily observed from Propositions~\ref{prop:dist2invariant} and \ref{prop:path2root} that, 
for a graph $G$ and a vertex $r\in V(G)$, 
if $G$ is factor-critical, 
then $(G, V(G)\setminus \{r\})$ is a graft with 
$\nu(G, V(G)\setminus \{r\}) = \nu(G, ( V(G)\setminus \{v \} ))$ for every $v\in V(G)\setminus \{r\}$.
In fact, the converse also holds according to Proposition~\ref{prop:critical2char}. 
This explains the validity of the definition of factor-critical grafts, 
where a pair of a factor-critical graph $G$ and its vertex set missing one vertex 
can be considered a ``graphic'' factor-critical graft. 
Proposition~\ref{prop:critical2char} further proves that 
every factor-critical graft is in fact  ``graphic''.

Proposition~\ref{prop:critical2char} implies that 
if $(G, T)$ is a factor-critical graft with root $r$, then, for every $v\in V(G)\setminus \{r\}$, 
$(G, T \Delta \{r, v\})$ is also a factor-critical graft with root $v$. 
Proposition~\ref{prop:critical2char} also implies the next property.

\begin{observation} \label{obs:crgraft2matching} 
Let $(G, T)$ be a factor-critical graft with root $r$, and let $F$ be a minimum join of $(G, T)$. 
Then,  $|\parcut{G}{v}\cap F| = 1$ for every $v\in V(G)\setminus \{r\}$, 
whereas $\parcut{G}{r} \cap F = \emptyset$.  
\end{observation}

From Proposition~\ref{prop:critical2char}, 
it is obvious that Theorem~\ref{thm:lovasz} can easily be converted to a characterization for factor-critical grafts. 

\begin{definition} 
Define the family $\fcgtr{r}$ of grafts with vertex $r$ as follows. 
\begin{rmenum} 
\item  $( (\{r\},  \emptyset), \emptyset)$ is a member of $\fcgtr{r}$. 
\item Let $(G, T)$ be a member of $\fcgtr{r}$, and let $(P, T')$ be an ear graft relative to $(G, T)$ with bonds $s$ and $t$ 
such that $P$ is round, $|E(P)|$ is odd, and $V(P)\cap T = V(P) \setminus \{s, t\}$. 
Then, $(G, T) \oplus (P, T')$ is a member of $\fcgtr{r}$. 
\end{rmenum} 
\end{definition} 

\begin{theorem} \label{thm:lovaszgraft} 
A graft is factor-critical with root $r$ if and only if it is a member of $\fcgtr{r}$. 
\end{theorem} 

\begin{remark} 
It is again easily observed from the definition, Proposition~\ref{prop:critical2char}, or Theorem~\ref{thm:lovaszgraft} 
that every factor-critical graft is nonbipartite. 
\end{remark}

\section{Bipartite Grafts with Ordered Bipartition} \label{sec:bipartite}

We now define bipartite grafts with an ordered bipartition 
and formulate related fundamental concepts  such as addition and ear grafts. 
From this section onward, we only discuss bipartite grafts with an ordered bipartition.

\begin{definition} 
We call a quadruple $(G, T; A, B)$ a {\em bipartite graft with an ordered bipartition}  or sometimes simply a (bipartite) graft 
if  $(G, T)$ is a bipartite graft with color classes $A$ and $B$. 
Note that $(G, T; A, B)$ and $(G, T; B, A)$ are not equivalent as bipartite grafts with an ordered bipartition. 
We often refer to a bipartite graft with an ordered bipartition simply as a bipartite graft or a graft. 
That is, 
if we say that $(G, T; A, B)$ is a (bipartite) graft, we imply that $(G, T; A, B)$ is a bipartite graft with an ordered bipartition. 
\end{definition}

Additional conditions are required for bipartite grafts with an ordered bipartition regarding addition and ear grafts.

\begin{definition} 
Let $(G_1, T_1; A_1, B_1)$ and $(G_2, T_2; A_2, B_2)$ be bipartite grafts with $A_1 \cap B_2 = \emptyset$ and $A_2 \cap B_1 = \emptyset$.  
We define the sum of $(G_1, T_1; A_1, B_1)$ and $(G_2, T_2; A_2, B_2)$ as 
the bipartite graft $(G_1 + G_2, T_1\Delta T_2; A_1 \cup A_2, B_1 \cup B_2)$ 
and denote this graft by $(G_1, T_1; A_1, B_1)\oplus(G_2, T_2; A_2, B_2)$. 
\end{definition}

\begin{definition} 
Let $(G, T; A, B)$ and $(P, T'; A', B')$ be bipartite grafts with an ordered bipartition. 
We say that $(P, T'; A', B')$ is an {\em ear graft}  relative to $(G, T; A, B)$ if     
\begin{rmenum} 
\item $A \cap B' = \emptyset$ and $A' \cap B = \emptyset$, and 
\item  $(P, T')$ is an ear graft relative to $(G, T)$. 
\end{rmenum} 
\end{definition}

That is, if $(P, T'; A', B')$ is an  ear graft   relative to $(G, T; A, B)$, 
then $(G, T; A, B) \oplus (P, T'; A', B')$ is a bipartite graft 
$( G+P, T\Delta T; A\dot{\cup} A', B \dot{\cup} B')$.

\section{Combs}  \label{sec:comb}

We present the concept of comb bipartite grafts 
and a sufficient condition for a bipartite graft to be a comb 
that has been provided by Seb\"o~\cite{DBLP:journals/jct/Sebo90}. 
Combs are also discussed by Kita~\cite{kita2017parity}. 

\begin{definition} 
A bipartite graft $(G, T; A, B)$ is a {\em comb} if $B\subseteq T$ holds and $\nu(G, T) = |B|$. 
For a comb $(G, T; A, B)$, $A$ and $B$ are called the {\em spine} and {\em tooth} sets.  
\end{definition} 

\begin{remark}
It is easily confirmed that the definition of combs can be rephrased as follows. 
Let $(G, T; A, B)$  be a graft, and let $F$ be a minimum join. 
Then, $(G, T; A, B)$ is a comb if and only if $|\parcut{G}{v}\cap F| = 1$ for every $v\in B$.  
\end{remark}

\begin{theorem}[Seb\"o~\cite{DBLP:journals/jct/Sebo90}]  \label{thm:dist2comb} 
Let $(G, T; A, B)$ be a bipartite graft, and let $r\in A$. 
Let $F$ be a minimum join.  
If $\distgtf{G}{T}{F}{r}{v} = 0$ for every $v\in A$ and 
$\distgtf{G}{T}{F}{r}{v} = -1$ for every $v\in B$, 
then $(G, T; A, B)$ is a comb. 
\end{theorem} 

We use Theorem~\ref{thm:dist2comb} to prove Theorem~\ref{thm:cr2qcomb}.

\section{Quasicombs} \label{sec:quasicomb}

We now introduce the new concept of quasicombs. 
This is a broader concept than that of combs.    
Later in this paper, 
we prove that the bipartite graft analogue of factor-critical graphs forms a special class of quasicombs.

\begin{definition} 
Let $(G, T; A, B)$ be a bipartite graft. 
If $\nu(G, T) = |B\cap T|$, then this graft is called a {\em quasicomb}, 
and sets $A$ and $B$ are called the {\em spine} and {\em tooth} sets, respectively. 
\end{definition} 

\begin{remark} 
The definition of quasicombs can be rephrased as follows. 
Let $(G, T; A, B)$ be a bipartite graft, and let $F$ be a minimum join. 
Then, $(G, T; A, B)$ is a quasicomb if and only if $|\parcut{G}{v} \cap F | \le 1$ for every $v\in B$. 
\end{remark}

Lemmas~\ref{lem:qcomb2dist} and \ref{lem:qcomb2path} below are fundamental observations regarding quasicombs 
and can easily be confirmed. 
We  use these lemmas in later sections, sometimes without explicitly mentioning them.

\begin{lemma} \label{lem:qcomb2dist} 
Let $(G, T; A, B)$ be a quasicomb, and let $F$ be a minimum join of $(G, T; A, B)$. 
Then, the following properties hold: 
\begin{rmenum} 
\item $\distgtf{G}{T}{F}{x}{y} \ge 0$ for every $x, y \in A$. 
\item $\distgtf{G}{T}{F}{x}{y} \ge -1$ for every $x\in A$ and every $y\in B$. 
\item $\distgtf{G}{T}{F}{x}{y} \ge -2$ for every $x, y\in B$.  
\end{rmenum} 
\end{lemma}

\begin{definition}
Let $(G, T; A, B)$ be a graft, and let $F \subseteq E(G)$ and $s, t \in V(G)$.  
We say that a path $P$ between $s$ and $t$ is {\em $F$-balanced} 
if $| \parcut{P}{v} \cap F | = 1$ holds for every $v\in V(P)\cap B  \setminus \{s, t\}$.   
\end{definition}

\begin{lemma} \label{lem:qcomb2path} 
Let $(G, T; A, B)$ be a quasicomb, and let $F$ be a minimum join of $(G, T; A, B)$.  
Let $x, y\in V(G)$, and let $P$ be an $F$-balanced path between $x$ and $y$.    
Then, the following hold: 
\begin{rmenum} 
\item If $x, y\in A$ holds, then $w_F(P ) = 0$. 
\item If $x\in A$ and $y\in B$ hold, then $w_F(P) \in \{ 1, -1 \}$ holds. Additionally, 
  $w_F(P)$ is $-1$ or $1$  if and only if  the edge of $P$ connected to $y$ is in $F$ or not in $F$, respectively.
\item If $x, y\in B$ hold, then $w_F(P)  \in \{ -2, 0, 2\}$ holds. 
Additionally, $w_F(P)$ is equal to $-2$ or $2$  if and only if  the edges of $P$ connected to the ends are in $F$ or not in $F$, respectively.  
\end{rmenum} 
\end{lemma}

\section{Critical Quasicombs}  \label{sec:crqcomb}

We now define the new concept of critical bipartite grafts as a bipartite analogue of  factor-critical grafts 
and a bipartite analogue of factor-critical grafts.   
We also prove from Theorem~\ref{thm:dist2comb} that critical bipartite grafts are in fact quasicombs.

\begin{definition}
Let $(G, T; A, B)$ be a bipartite graft. 
Let $r\in B$.  
We say that $(G, T; A, B)$ is {\em critical } with {\em root } $r$ if 
$\nu(G, T\Delta \{x, r\}) - \nu(G,T) = 1$ for each $x \in A$ and $\nu(G, T\Delta \{x, r\}) - \nu(G,T) = 0$ for each $x\in B$. 
\end{definition} 

\begin{remark}
The rationality of this definition can be explained from the definition of factor-critical grafts. 
A simple parity argument proves that, for a minimum join $F$, the $F$-distance between $r$ and a vertex $x$ 
is odd or even if $x$ is in $A$ or $B$, respectively. 
Therefore, 
$\nu(G, T\Delta \{x, r\}) - \nu(G,T)$ is odd for every $x\in A$ and is even for every $x\in B$. 
Furthermore, every factor-critical graph or graft is connected. 
Critical bipartite grafts can equivalently be defined as 
connected bipartite grafts $(G, T; A, B)$ with $\nu(G, T\Delta \{x, r\})  - \nu(G,T) = 0$ for each $x\in B$.  
\end{remark} 

In the sequel, we frequently use the following three lemmas, sometimes without explicitly mentioning them. 
First, by Proposition~\ref{prop:dist2invariant},  the definition of critical quasicombs can be rephrased as follows. 

\begin{lemma} \label{lem:critical2dist} 
Let $(G, T; A, B)$ be a bipartite graft, and let $r\in B$. 
Let $F$ be a minimum join. 
Then, the following properties are equivalent. 
\begin{rmenum} 
\item $(G, T; A, B)$ is a critical bipartite graft with root $r$. 
\item $\distgtf{G}{T}{F}{x}{r} = 1$ for each $x \in A$, and $\distgtf{G}{T}{F}{x}{r} = 0$ for each $x\in B$. 
\end{rmenum} 
\end{lemma}

The next theorem states that every critical bipartite graft is in fact a quasicomb.

\begin{theorem} \label{thm:cr2qcomb}
If $(G, T; A, B)$ is a critical bipartite graft with root $r\in V(G)$, 
 then $(G, T; A, B)$ is a quasicomb  with $T \cap B = B \setminus \{r\}$. 
\end{theorem} 
\begin{proof} 
Let $\hat{G}$ be the graph obtained from $G$ by adding a new vertex $s$ and an edge $sr$. 
Let $(\hat{G}, \hat{T}; \hat{A}, \hat{B})$ be the bipartite graft
with  $\hat{T} = T \cup \{s, r\}$, $\hat{A} = A \cup \{s\}$, and $\hat{B} = B$.  
Let $F$ be a minimum join of $(G, T)$, and let $\hat{F} = F \cup \{sr\}$. 
Then, $\hat{F}$ is obviously a minimum join of $(\hat{G}, \hat{T})$ by Lemma~\ref{lem:minimumjoin}. 
 Lemma~\ref{lem:critical2dist} easily implies   $\distgtf{\hat{G}}{\hat{T}}{\hat{F}}{s}{v} = 0$ for every $v\in \hat{A}$ 
and $\distgtf{\hat{G}}{\hat{T}}{\hat{F}}{s}{v} = -1$ for every $v\in \hat{B}$. 
Then, Theorem~\ref{thm:dist2comb} implies that $(\hat{G}, \hat{T}; \hat{A}, \hat{B})$ is a comb. 
This further implies $|\hat{F}| = |\hat{B}| = |\hat{B}\cap \hat{T}|$. 
Therefore, $|F| = |B\cap T|$ holds, and $(G, T; A, B)$ is a quasicomb with $T \cap B = B \setminus \{r\}$. 
\end{proof}

In the remainder of this paper, we assume Theorem~\ref{thm:cr2qcomb} without explicitly mentioning it. 
Hereafter, we refer to critical bipartite grafts as critical quasicombs and refrain from referring to Theorem~\ref{thm:cr2qcomb}. 
We can also confirm from Theorem~\ref{thm:cr2qcomb} that the bipartite analogue of Observation~\ref{obs:crgraft2matching} 
holds for critical bipartite grafts. 

From Lemma~\ref{lem:critical2dist},   the next lemma is easily implied.

\begin{lemma}  \label{lem:critical} 
Let $(G, T; A, B)$ be a critical quasicomb whose root is $r\in B$. 
Let $F$ be a minimum join of $(G, T)$. 
Then, the following properties hold. 
\begin{rmenum} 
\item $\parcut{G}{r} \cap F = \emptyset$. 
\item $|\parcut{G}{v} \cap F| = 1$ for each $v\in B\setminus \{r\}$. 
\item Every $F$-shortest path between $v\in B \setminus \{r\}$ and $r$
is an $F$-balanced path of weight $0$ in which $v$ is connected to an edge from $F$.  
\item Every $F$-shortest path between $v\in A$ and $r$ is an $F$-balanced path of weight $1$.  
\end{rmenum} 
\end{lemma}

\section{Effective Ear Grafts} \label{sec:ear}

In this section, we define the concept of effective ear grafts for bipartite grafts with an ordered bipartition 
and then present some of their basic properties. 
In the next section, critical quasicombs are characterized as grafts that are constructed from effective ear grafts.  
That is, effective ear grafts are a bipartite graft analogue of round ears with an odd number of edges that appear in Theorem~\ref{thm:lovasz}, 
and they also serve as  a graft analogue of ear grafts that appear in Theorem~\ref{thm:lovaszgraft}.

\begin{definition} 
Let $(G, T; A, B)$  be a graft, and let $(P, T'; A', B')$ be an ear graft relative to $(G, T; A, B)$ 
for which $s, t \in V(P)$ are bonds. 
We say that $(P, T'; A', B')$ is {\em effective} if it satisfies the following. 
\begin{rmenum}
\item $( V(P)\setminus \{s, t\}) \cap B \subseteq T$. 
\item $\{s, t\} \cap B \cap T = \emptyset$. 
\item If it is straight, then $s\in A'$ implies $s\in T$, 
whereas $v\in A'$ implies $v\not\in T$, where $v$ is the free end of $P$. 
\end{rmenum} 
\end{definition}

The next lemma can easily be confirmed; 
it claims that effective straight ears can be classified into four types with simple structures.

\begin{lemma} \label{lem:lear2matching} 
Let $(G, T; A, B)$  be a graft, and 
let $(P, T'; A', B')$ be an ear graft relative to $(G, T; A, B)$  that is effective and straight. 
Let $s$ and $t$ be the free end and the bond of $P$, respectively. 
Then,  $V(P)\setminus \{s, t\} \subseteq T$ holds. 
Additionally, $(P, T'; A, B)$ has only one minimum join $F$, and it satisfies the following properties: 
\begin{rmenum} 
\item  If $s\in A'$ and $t\in B'$ hold, then $F$ is a perfect matching of $P - s - t$. 
\item  If $s\in A'$ and $t\in A'$ hold, then $F$ is a perfect matching of $P - s$. 
\item If $s\in B'$ and $t\in A'$ hold, then $F$ is a perfect matching of $P$. 
\item If $s\in B'$ and $t\in B'$ hold, then $F$ is a perfect matching of $P - t$. 
\end{rmenum} 
\end{lemma}

The next lemma is easily confirmed from the definition of effective ear grafts and Lemmas~\ref{lem:qcomb2path} and \ref{lem:lear2matching}; 
it is used to prove Lemma~\ref{lem:ear2critical}.

\begin{lemma} \label{lem:ear2path} 
Let $(G, T; A, B)$  be a graft, and 
let $(P, T'; A', B')$ be an effective ear graft relative to $(G, T; A, B)$ with bonds $s$ and $t$. 
Let  $F$ be a minimum join of $(P, T'; A', B')$. 
Then, the following properties hold: 
\begin{rmenum} 
\item For each $x\in V(P)\cap A'$, there exists $r \in \{s, t\}$ with either $r\in A'$ or $r\in B'$ 
such that $xPr$ is an $F$-balanced path of weight $0$,  or $xPr$ is an $F$-balanced path of weight $1$, respectively. 
\item For each $x\in V(P)\cap B'$,   
 there exists $r \in \{s, t\}$ with either $r\in A'$  or $r\in B'$  such that 
  $xPr$ is an $F$-balanced path of weight $-1$, 
or   $xPr$ is an $F$-balanced path of weight $0$ in which $x$ is connected to an edge from $F$, respectively.  
\end{rmenum} 
\end{lemma}

\section{Constructive Characterization for Critical Quasicombs}  \label{sec:main} 

We finally present our main results. 
In this section, we provide a constructive characterization for critical quasicombs.  
The main results here, Theorems~\ref{thm:critical2char}  and \ref{thm:ear2balanced},  
are used by Kita~\cite{kita2021bipartite} for deriving the generalization of the Dulmage-Mendelsohn decomposition for bipartite grafts. 

\begin{definition} 
We define a family $\critical{r}$ of bipartite grafts with an ordered bipartition that have a vertex $r$ as follows. 
\begin{rmenum} 
\item The graft $( (\{r\}, \emptyset), \emptyset; \emptyset, \{r\})$ is a member of $\critical{r}$. 
\item Let $(G, T; A, B) \in \critical{r}$, and let $(P, T'; A', B')$ be an effective ear graft relative to $(G, T; A, B)$. 
Then, $(G, T; A, B) \oplus (P, T'; A', B')$  is a member of $\critical{r}$. 
\end{rmenum} 
\end{definition} 

In the following,  we show Lemmas~\ref{lem:ear2critical}, \ref{lem:critical2increment}, and \ref{lem:critical2ear}  
 to prove in Theorem~\ref{thm:critical2char} that $\critical{r}$ is the family of critical quasicombs with root $r$.

\begin{lemma} \label{lem:ear2critical} 
Every member of $\critical{r}$ is a critical quasicomb with root $r$. 
\end{lemma} 
\begin{proof} We prove this lemma by induction along the definition of $\critical{r}$. 
First, the statement clearly holds for $( (\{r\}, \emptyset), \emptyset; \emptyset, \{r\} )$. 
Next, let $(G, T; A, B) \in \critical{r}$, 
and assume that the statement holds for $(G, T; A, B)$; that is, $(G, T; A, B)$ is a critical quasicomb with root $r$. 
Let $(P, T'; A', B')$ be an ear graft relative to $(G, T)$. 
Let $F$ and $F'$ be minimum joins of $(G, T)$ and $(P, T')$, respectively. 
Let $(\hat{G}, \hat{T}; \hat{A}, \hat{B}):= (G, T; A, B) \oplus (P, T'; A', B')$. 
Let $\hat{F} := F \cup F'$. 
As is mentioned in Observation~\ref{obs:addition}, $\hat{F}$ is a join of $(\hat{G}, \hat{T}; \hat{A}, \hat{B})$. 

\begin{pclaim} \label{claim:join} 
The graft $(\hat{G}, \hat{T}; \hat{A}, \hat{B})$ is a quasicomb, 
and $\hat{F}$ is a minimum join.  
\end{pclaim} 
\begin{proof} 
We first prove that $\hat{F}$ is a minimum join. 
According to Lemma~\ref{lem:minimumjoin}, we only need to prove that there is no circuit with negative $\hat{F}$-weight. 
If $C$ is a circuit in $(\hat{G}, \hat{T}; \hat{A}, \hat{B})$ with negative $\hat{F}$-weight, 
then $P$ is required to be a round ear with distinct bonds $s$ and $t$, 
and $C$ is the sum of $P$ and a path $Q$ in $(G, T; A, B)$ between $s$ and $t$. 
Hence, the claim can be proved if we prove $w_{F'}(P) + \distgtf{G}{T}{F}{s}{t} \ge 0$ 
under the supposition that  $P$ is a round ear with distinct bonds $s$ and $t$. 

First, consider the case with $s, t\in B$. 
Lemma~\ref{lem:qcomb2path} implies  $w_{\hat{F}}(P) = 2$, whereas Lemma~\ref{lem:qcomb2dist} implies $\distgtf{G}{T}{F}{s}{t} \ge -2$.  
Hence, $w_{F'}(P) + \distgtf{G}{T}{F}{s}{t} \ge 0$.   
For the other cases, where the bonds $s$ and $t$ are both in $A$ or individually in $A$ and $B$,  
similar discussions prove $w_{F'}(P) + \distgtf{G}{T}{F}{s}{t} \ge 0$.   
Therefore, no circuit in $(\hat{G}, \hat{T}; \hat{A}, \hat{B})$ can be of negative $\hat{F}$-weight; 
accordingly, $\hat{F}$ is a minimum join. 
This further implies that $\nu(\hat{G}, \hat{T}) = |\hat{B} \setminus \{r\}| = |\hat{B} \cap \hat{T}|$. 
That is, $(\hat{G}, \hat{T}; \hat{A}, \hat{B})$ is a quasicomb. 
Thus, the claim is proved. 
\end{proof}

Under Claim~\ref{claim:join}, we further prove that  $(\hat{G}, \hat{T}; \hat{A}, \hat{B})$ is a critical quasicomb.

\begin{pclaim} \label{claim:critical} 
$\distgtf{\hat{G}}{\hat{T}}{\hat{F}}{r}{x} = 1$  holds for every $x\in \hat{A}$, whereas   
$\distgtf{\hat{G}}{\hat{T}}{\hat{F}}{r}{x} = 0$  holds for every $x\in \hat{B}$. 

\end{pclaim} 
\begin{proof} 
Under the induction hypothesis and Lemma~\ref{lem:critical}, the claim obviously holds for every $x \in V(G)$. 
This further proves the claim for every $x\in V(P)\setminus V(G)$ from Lemma~\ref{lem:ear2path} 
by considering the concatenation of paths. 
\end{proof} 
Under Claims~\ref{claim:join} and \ref{claim:critical}, 
Lemma~\ref{lem:critical2dist} implies that $(\hat{G}, \hat{T}; \hat{A}, \hat{B})$ is a critical quasicomb with root $r$. 
This completes the proof of the lemma. 
\end{proof}

The next lemma is provided for proving Lemma~\ref{lem:critical2ear}. 

\begin{lemma} \label{lem:critical2increment} 
Let $(G, T; A, B)$ be a critical quasicomb with root $r\in B$, and let $F$ be a minimum join of $(G, T; A, B)$.  
Let $(G', T'; A', B')$ be a subgraft of $(G, T; A, B)$, where $A'\subseteq A$ and $B'\subseteq B$, 
such that $r\in B'$, $V(G') \subsetneq V(G)$, and $E_G[B', A\setminus A']\cap F = \emptyset$ hold.  
Then, $(G, T; A, B)$ has an ear graft $(G'', T''; A'', B'')$ relative to $(G', T'; A', B')$ 
with $V(G'')\setminus V(G') \neq \emptyset$ and $E_G[B'', A\setminus (A' \cup A'') ]\cap F = \emptyset$.  
\end{lemma} 
\begin{proof} First, we prove the case where $E_G[B', A\setminus A'] \neq \emptyset$. 
Let $e\in E_G[B', A\setminus A']$, and let $x\in A\setminus A'$ and $y\in B'$ be the ends of $e$. The assumption implies $e\not\in F$. 
Hence, $(G.e, \emptyset; \{x\}, \{y\})$ is an ear graft relative to $(G', T'; A', B')$ that meets the condition. 

Next, we prove the case where $E_G[B', A\setminus A'] = \emptyset$.  
Because $G$ is obviously connected, we have $E_G[A', B\setminus B'] \neq \emptyset$. 
Let $f\in E_G[A', B\setminus B']$, and let $u\in A'$ and $v\in B\setminus B'$ be the ends of $f$. 

First, consider the case with $f\not\in F$. 
Under Lemma~\ref{lem:critical2dist}, let $P$ be a path between $v$ and $r$ with $w_F(P) = 0$. 
Lemma~\ref{lem:critical} implies that $P$ is $F$-balanced, and the edge of $P$ connected to $v$ is in $F$. 
Hence, $f \not\in E(P)$ follows. 
Trace $P$ from $v$, and let $w$ be the first encountered vertex in $V(G')$. The assumption implies $w\in A'$. 
Let $Q := vPw + f$. 
Then, $( Q, T''; A\cap V(Q), B\cap V(Q) )$ is a desired ear graft relative to $(G', T'; A', B')$ that meets the condition, 
where $T''$ is the set of vertices from $V(Q)$ that are connected to an odd number of edges from $E(Q)\cap F$.

For the remaining case with $f\in F$, $(G. f, \{u, v\}; \{u\}, \{v\})$ is a desired ear graft relative to $(G', T'; A', B')$.  
This completes the proof. 
\end{proof}

Lemma~\ref{lem:critical2increment} implies the next lemma.

\begin{lemma} \label{lem:critical2ear} 
Every critical quasicomb with root $r$ is a member of $\critical{r}$. 
\end{lemma} 
\begin{proof} Let $(G, T; A, B)$ be a critical quasicomb with root $r$. 
We prove the lemma by induction on $|V(G)|$. 
If $| V(G) | = 1$, that is, $(G, T; A, B)$ is $( (\{r\}, \emptyset), \emptyset; \emptyset, \{r\})$, 
then  $(G, T; A, B)\in \critical{r}$ obviously holds. 
Next, let $|V(G)| > 1$, and assume that the statement holds for every critical quasicomb with a smaller number of vertices. 
Let $F$ be a minimum join of $(G, T; A, B)$. 
Let $(G', T', A', B')$ be a maximal subgraft of $(G, T; A, B)$ 
that is a member of $\critical{r}$  for which $E_G[B', A\setminus A'] \cap F = \emptyset$.

Suppose  $V(G') \subsetneq V(G)$. 
Lemma~\ref{lem:ear2critical} implies that $(G', T'; A', B')$ is a critical quasicomb with root $r$. 
Therefore, Lemma~\ref{lem:critical2increment} implies that 
$(G, T; A, B)$ has an ear graft $(P, T''; A'', B'')$ relative to $(G', T'; A', B')$ 
such that $V(P)\setminus V(G') \neq \emptyset$ and $E_G[B'', A\setminus (A' \cup A'')]\cap F = \emptyset$.   
Then, $(G', T'; A', B') \oplus (P, T''; A'', B'')$ is a member of $\critical{r}$ for which $E_G[B'\cup B'', A\setminus (A' \cup A'')]\cap F = \emptyset$.  
This contradicts the maximality of $(G', T'; A', B')$. 
Hence, we obtain $V(G') = V(G)$. 
This further implies $(G, T; A, B) \in \critical{r}$.  
The lemma is proved. 
\end{proof} 
 
Combining Lemmas~\ref{lem:ear2critical} and \ref{lem:critical2ear}, 
we now obtain Theorem~\ref{thm:critical2char}, the constructive characterization for critical quasicombs. 

\begin{theorem} \label{thm:critical2char} 
A bipartite graft with an ordered bipartition  is a critical quasicomb with root $r$ 
if and only if it is a member of $\critical{r}$. 
\end{theorem}

We next provide Theorem~\ref{thm:ear2balanced}.

\begin{definition}
Let $(G, T; A, B) \in \critical{r}$. 
Let $\{ (G_i, T_i; A_i, B_i) \in \critical{r} : i = 1, \ldots, l\}$, where $l \ge 1$, be a family of grafts such that  
\begin{rmenum} 
\item $(G_1, T_1; A_1, B_1) = ( ( \{r\}, \emptyset), \emptyset; \emptyset, \{r\} )$, 
\item for each $i \in \{1, \ldots, l\} \setminus \{1\}$, 
$(G_{i+1}, T_{i+1}; A_{i+1}, B_{i+1}) = (G_i, T_i; A_i, B_i) \oplus (P_i, T_i'; A_i', B_i')$, 
where $(P_i, T_i'; A_i', B_i')$ is an effective ear graft relative to $(G_i, T_i; A_i, B_i)$, and 
\item $(G_l, T_l; A_l, B_l) = (G, T; A, B)$. 
\end{rmenum} 
We call the family $\{ (P_i, T_i'; A_i', B_i'): i= 1,\ldots, l\}$ a {\em graft ear decomposition} of $(G, T; A, B)$. 
\end{definition} 

\begin{definition} 
Let $(G, T; A, B)$ be a critical quasicomb with root $r\in B$, and let $F \subseteq E(G)$. 
Let  $\mathcal{P}$ be a family of bipartite grafts with 
 an  ordered bipartition that is a graft ear decomposition of $(G, T; A, B)$. 
We say that $\mathcal{P}$ is {\em $F$-balanced} if
 $F\cap E(P)$ is a minimum join for each $(P, T'; A', B') \in \mathcal{P}$. 
\end{definition}  

The next statement clearly follows from the proof of Lemma~\ref{lem:critical2ear}. 

\begin{theorem} \label{thm:ear2balanced}  
Let $(G, T; A, B)$ be a critical quasicomb with root $r\in B$ and $F$ be a minimum join.   
Then, there is a graft ear decomposition of $(G, T; A, B)$ that is $F$-balanced. 
\end{theorem}

\begin{figure} 
\includegraphics[width=.7\textwidth]{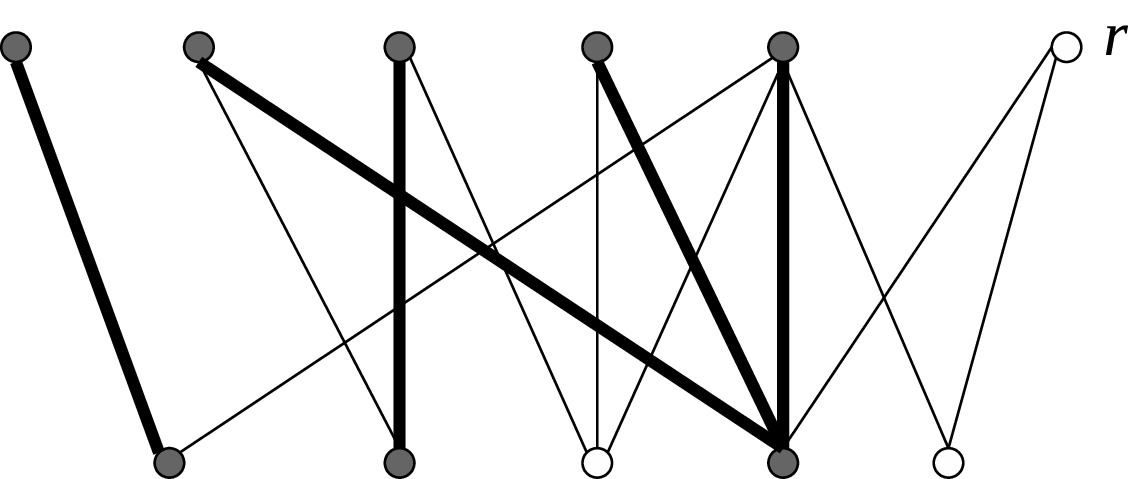} 
\caption{Critical quasicomb $(G, T; A, B)$ with root $r$ and one of its minimum joins $F$: 
The gray and white points indicate vertices in $T$ and $V(G)\setminus T$, respectively, 
and the upper and lower  vertices form color classes $B$ and $A$, respectively. 
Bold lines indicate edges in $F$.} 
\label{fig:crqcomb} \label{fig:join} 
\end{figure}

\begin{figure} 
\includegraphics[width=.3\textwidth]{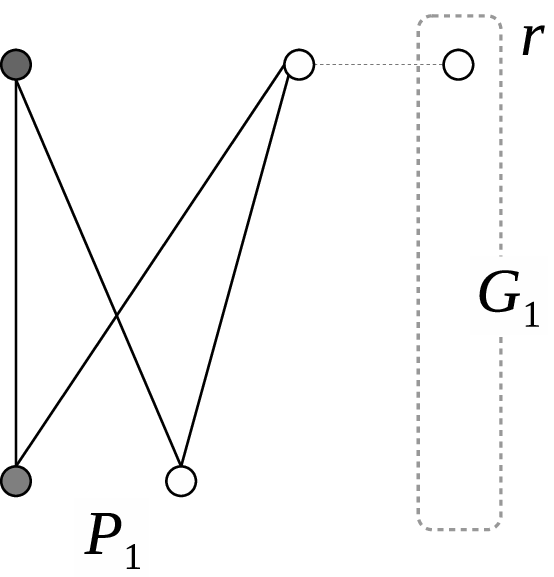} 
\caption{Graft $(G_1, T_1; A_1, B_1)$ and  ear graft $(P_1, T_1'; A_1', B_1')$. } 
\label{fig:g1p1} 
\end{figure} 

\begin{figure} 
\includegraphics[width=.4\textwidth]{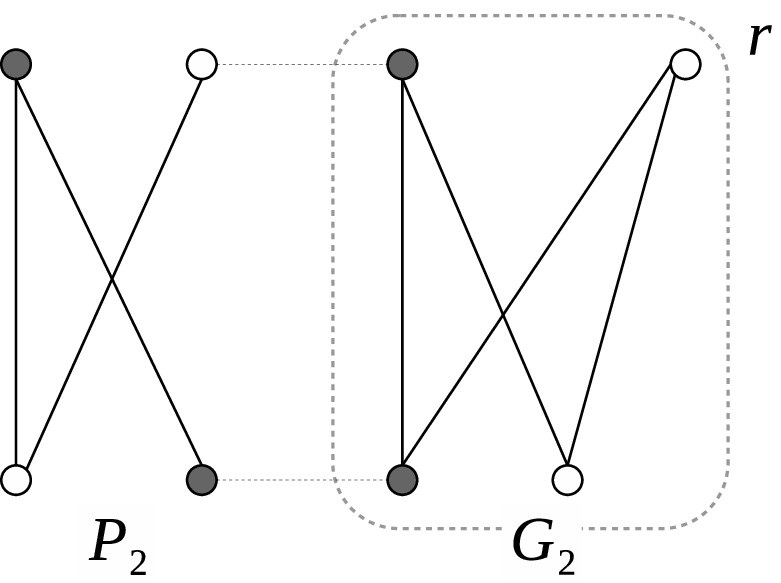} 
\caption{Graft $(G_2, T_2; A_2, B_2)$ and  ear graft $(P_2, T_2'; A_2', B_2')$. } 
\label{fig:g2p2} 
\end{figure} 

\begin{figure} 
\includegraphics[width=.7\textwidth]{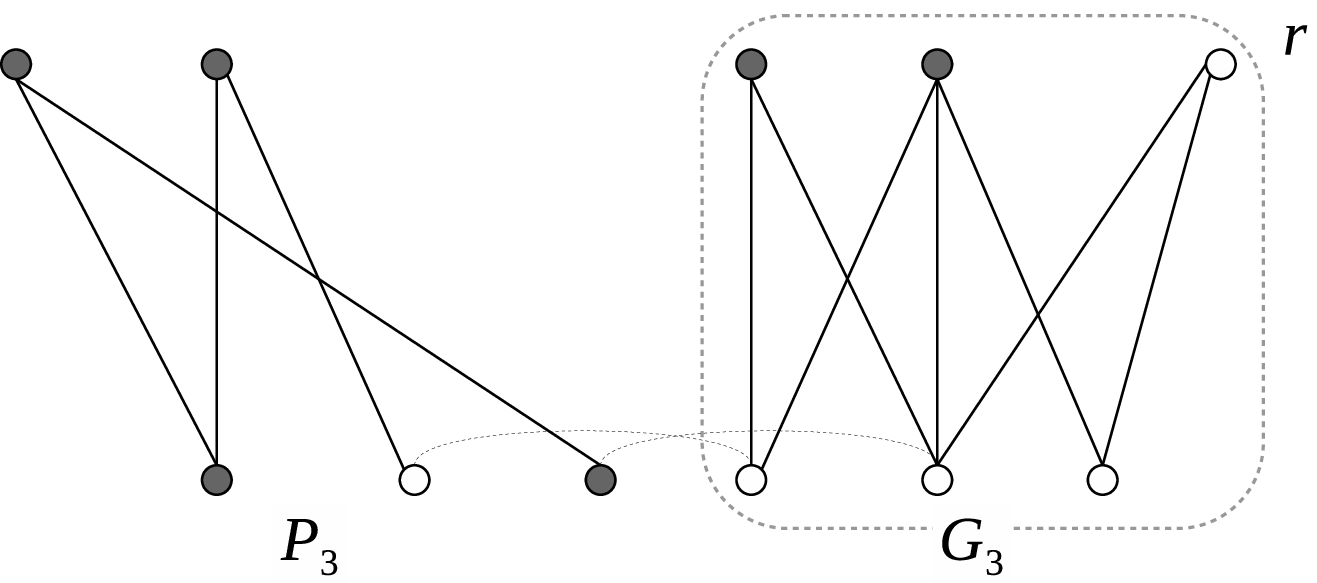} 
\caption{Graft $(G_3, T_3; A_3, B_3)$ and  ear graft $(P_3, T_3'; A_3', B_3')$. } 
\label{fig:g3p3} 
\end{figure} 

\begin{figure} 
\includegraphics[width=.7\textwidth]{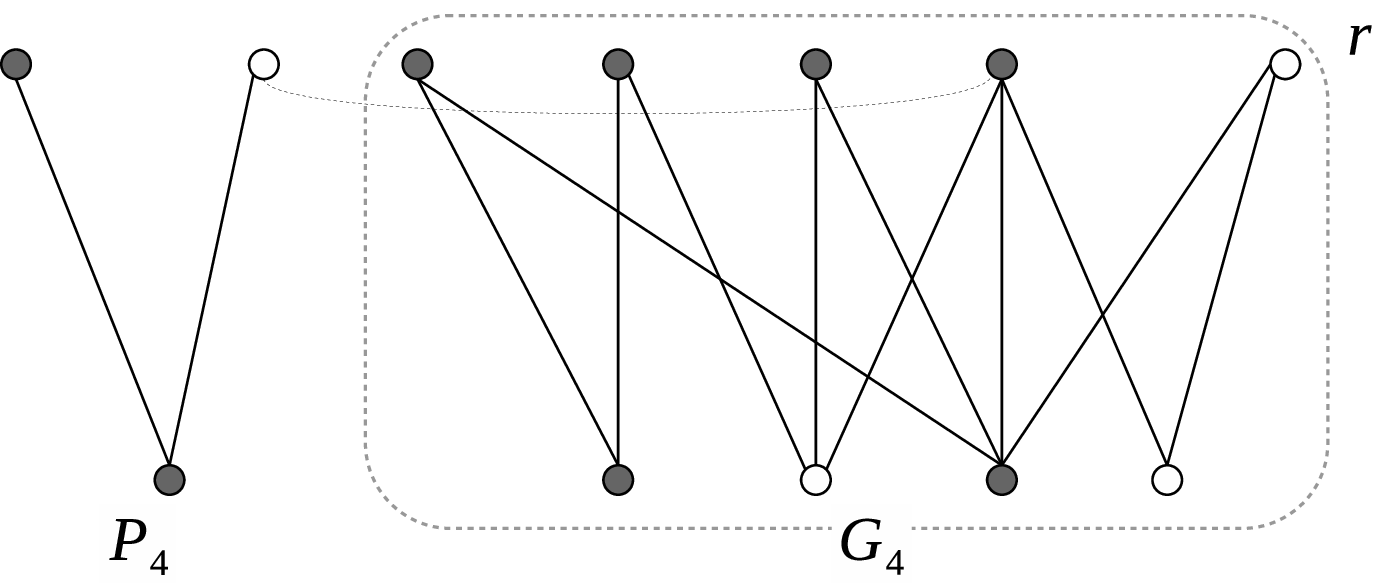} 
\caption{Graft $(G_4, T_4; A_4, B_4)$ and  ear graft $(P_4, T_4'; A_4', B_4')$. } 
\label{fig:g4p4} 
\end{figure}

\begin{example} 
Figures~\ref{fig:g1p1} to \ref{fig:g4p4} illustrate an ear decomposition $\{(P_i, T_i'; A_i', B_i'): i = 1,\ldots, 4\}$  
of the critical quasicomb $(G, T; A, B)$  from Figure~\ref{fig:join}.  
For each $i \in \{1,\ldots, 4\}$,  $(G_i, T_i; A_i, B_i)$ 
is a critical quasicomb with root $r$, and $(P_i, T_i'; A_i', B_i')$ is an ear graft relative to it. 
It holds that  $(G_{i+1}, T_{i+1}; A_{i+1}, B_{i+1}) = (G_i, T_i; A_i, B_i) \oplus (P_i, T_i'; A_i', B_i')$,  
 where $(G_5, T_5; A_5, B_5) := (G, T; A, B)$. 
As in Figure~\ref{fig:join}, 
lower points indicate vertices in $A_i$ or $A_i'$, whereas upper points indicate vertices in $B_i$ or $B_i'$. 
Gray vertices indicate those from $T_i$ or $T_i'$, whereas white vertices indicate the others.  
Each dotted line shows a pair of vertices that are identified through the addition. 
As can be seen, this ear decomposition is $F$-balanced with respect to the minimum join $F$ from Figure~\ref{fig:join}. 
\end{example}

\section{Conclusion} \label{sec:conclusion} 

We introduced critical bipartite grafts as a bipartite graft analogue of factor-critical graphs 
and provided a constructive characterization of this class of grafts, 
which is an analogue of Lov\'asz' theorem~\cite{lovasz1972note, lp1986, kv2008, schrijver2003} for factor-critical graphs. 
Our main results will be used to derive  
a generalization of the Dulmage-Mendelsohn canonical decomposition for grafts~\cite{kita2021bipartite}.

\begin{ac} 
This study was supported by JSPS KAKENHI Grant Number 18K13451. 
\end{ac}

\bibliographystyle{splncs03.bst}
\bibliography{tbicath.bib}

\end{document}